\documentclass[preprint,12pt]{elsarticle}

\usepackage{amssymb}
\usepackage{amsthm}

\usepackage{amsmath}

\newtheorem{theorem}{Theorem}[section]
\newtheorem{lemma}{Lemma}[section]
\newtheorem{corollary}{Corollary}[section]

\journal{arXiv}

\begin{document}

\begin{frontmatter}


\title{On the mean average of integer partition as the sum of powers} 
\author{Pengyong Ding\corref{cor1}\fnref{label2}}
 \cortext[cor1]{School of Mathematics, Shandong University, Jinan, Shandong 250100 China}





\begin{abstract}
This paper is concerned with the function $r_{k,s}(n)$, the number of (ordered) representations of $n$ as the sum of $s$ positive $k$-th powers, where integers $k,s\ge 2$. We examine the mean average of the function, or equivalently,
\begin{equation*}
\sum_{m=1}^n r_{k,s}(m).
\end{equation*}


\end{abstract}




\end{frontmatter}



\section{Introduction}

This paper is concerned with the function $r_{k,s}(n)$, the number of (ordered) representations of $n$ as the sum of $s$ positive $k$-th powers, where integers $k,s\ge 2$:
\begin{equation}
r_{k,s}(n)=\sum_{\substack{
x_1, x_2,\cdots x_s \\
x_1^k+x_2^k+\cdots+x_s^k=n
}} 1.
\end{equation}
The function is important in the study of Waring's problem, which is to find the least $s$ for a given $k$, such that $r_{k,s}(n)>0$ for every sufficiently large $n$. Besides Waring's problem, the function is also useful in the study of Diophantine equations. However, the property of $r_{k,s}(n)$ is only poorly understood even when $k$ and $s$ are small. For example, when both $k$ and $s$ are as small as 3, the asymptotic formula for the second moment sum of $r_{3,3}(n)$ is still unknown. The best result that we have already known is provided by Vaughan \cite{Va1}:
\begin{equation*}
\sum_{m=1}^n r_{3,3}(n)^2 \ll n^{\frac{7}{6}}(\log n)^{\varepsilon-\frac{5}{2}},
\end{equation*}
although we may conjecture that
\begin{equation*}
\sum_{m=1}^n r_{3,3}(n)^2 \sim Cn
\end{equation*}
for some positive constant $C$. The mean average of $r_{k,s}(n)$, or equivalently, the sum
\begin{equation}\label{sumimportant}
\sum_{m=1}^n r_{k,s}(m),
\end{equation}
is not as badly known as the other properties. In fact, similar to the argument by Vaughan \cite{Va1}, if we let
\begin{equation*}
\Delta_{k,s}(n)=\sum_{m=1}^n r_{k,s}(m) - \frac{\Gamma \big( \frac{k+1}{k} \big)^s}{\Gamma\big( \frac{k+s}{k} \big)} n^{\frac{s}{k}},
\end{equation*}
then it represents the difference between the number of lattice points in an $s$-dimensional convex body and its volume, which is bounded by its $(s-1)$-dimensional surface volume. So we have
\begin{equation*}
\Delta_{k,s}(n) \ll n^{\frac{s-1}{k}},
\end{equation*}
and thus
\begin{equation}\label{original}
\sum_{m=1}^n r_{k,s}(m) =  \frac{\Gamma \big( \frac{k+1}{k} \big)^s}{\Gamma\big( \frac{k+s}{k} \big)} n^{\frac{s}{k}} + O\big(n^{\frac{s-1}{k}}\big).
\end{equation}
However, we can improve that result by applying van der Corput's method. For example, Vaughan \cite{Va1} have shown that
\begin{equation*}
\sum_{m=1}^n r_{3,2}(m) = \frac{\Gamma( \frac{4}{3} )^2}{ \Gamma( \frac{5}{3} )} n^{\frac{2}{3}}+O\Big(n^{\frac{2}{9}} (\log n)^{\frac{1}{3}}\Big),
\end{equation*}
and
\begin{equation*}
\sum_{m=1}^n r_{3,3}(m) = \Gamma\Big(\frac{4}{3} \Big)^3 n - \frac{\Gamma( \frac{4}{3} )^2}{ 2\Gamma( \frac{5}{3} )} n^{\frac{2}{3}} + O\Big( n^{\frac{5}{9}} (\log n)^{\frac{1}{3}} \Big), 
\end{equation*}
which is more precise than (\ref{original}) when $k=3$ and $s=2$ or $3$. Hence we can follow a similar argument and calculate the mean average of $r_{k,s}(n)$ when $k\ge 4$ and $s\ge 2$. The following theorem shows the result:
\begin{theorem}\label{MainThm}
If $k\ge4$ and $s$ are integers, and $2\le s\le k+1$, then
\begin{equation}\label{MainResult}
\sum_{m=1}^n r_{k,s}(m) = \frac{\Gamma \big( \frac{k+1}{k} \big)^s}{\Gamma\big( \frac{k+s}{k} \big)} n^{\frac{s}{k}} - \frac{s}{2}\cdot  \frac{\Gamma \big( \frac{k+1}{k} \big)^{s-1}}{\Gamma\big( \frac{k+s-1}{k} \big)} n^{\frac{s-1}{k}} +O\Big(n^{\frac{(s-1)k-1}{k^2}} \Big),
\end{equation}
or equivalently,
\begin{equation}\label{MainResultEq}
\sum_{m\le x} r_{k,s}(m) = \frac{\Gamma \big( \frac{k+1}{k} \big)^s}{\Gamma\big( \frac{k+s}{k} \big)} x^{\frac{s}{k}} - \frac{s}{2}\cdot  \frac{\Gamma \big( \frac{k+1}{k} \big)^{s-1}}{\Gamma\big( \frac{k+s-1}{k} \big)} x^{\frac{s-1}{k}} +O\Big(x^{\frac{(s-1)k-1}{k^2}} \Big).
\end{equation}
\end{theorem}

\section{The proof of Theorem \ref{MainThm}}

We state the following two results without proof. The first one is the van der Corput Lemma, which is Theorem 2.2 of Graham and Kolesnik \cite{GK}.
\begin{lemma}\label{vanderCorput}
Suppose that $a<b$ and $f$ has a continuous second derivative on $[a,b]$. Suppose also that $\mu>0$, $\eta>1$ and that for every $\alpha\in[a,b]$, we have $\mu\le |f^{''}(\alpha)| \le\eta\mu$. Then
\begin{equation*}
\sum_{a<n\le b} e(f(n))\ll \mu^{-\frac{1}{2}}+(b-a)\eta\mu^{\frac{1}{2}}.
\end{equation*}
\end{lemma}
The second result is Lemma 2.2 of Vaughan \cite{Va1}. For any $\alpha\in\mathbb{R}$, define
\begin{equation*}
B_1(\alpha)=\alpha-\lfloor\alpha\rfloor-\frac{1}{2}, \qquad B_2(\alpha)=\int_0^{\alpha} B_1(\beta)\mathrm{d}\beta,
\end{equation*}
then we have
\begin{lemma}\label{B1B2}
Let $H, \alpha \in \mathbb{R}$ and $H\ge 2$. Then
\begin{equation*}
B_1(\alpha)=-\sum_{0<|h|\le H}\frac{e(\alpha h)}{2\pi ih} +O\Big( \min \Big(1, \frac{1}{H\Vert\alpha\Vert}\Big) \Big)
\end{equation*}
and
\begin{equation*}
\min \Big( 1, \frac{1}{H\Vert\alpha\Vert}\Big)=\sum_{h=-\infty}^{\infty} c(h)e(\alpha h),
\end{equation*}
where
\begin{equation*}
c(0)=\frac{2}{H}\Big( 1+\log \frac{H}{2}\Big), \qquad c(h)\ll \min\Big(\frac{\log 2H}{H} , \frac{1}{|h|}, \frac{H}{h^2}\Big)(h\ne0).
\end{equation*}
\end{lemma}
We can use these results to prove the following lemma.
\begin{lemma}\label{lemmaerrorterm}
If $k\ge4$ is an integer and $x>1$, then
\begin{equation}\label{errorterm}
\sum_{m\le(x/2)^{1/k}} B_1 \Big( (x-m^k)^{\frac{1}{k}}\Big)\ll x^{\frac{k-1}{k^2}}.
\end{equation}
\end{lemma}
\begin{proof}
We use a similar method to the proof given by Vaughan \cite{Va1}. First, for convenience, we define
\begin{equation*}
M^{\prime}=\min\Big(2M, \Big(\frac{x}{2}\Big)^{\frac{1}{k}}\Big)
\end{equation*}
for every $M$. Then let $\nu$ be a parameter such that $1\le\nu\le(x/2)^{1/k}$, and define the set
\begin{equation*}
\mathcal{M}=\Big\{\nu\cdot 2^j \Big\vert j\ge0, \nu\cdot 2^j\le\Big(\frac{x}{2}\Big)^{\frac{1}{k}}\Big\}.
\end{equation*}
So we have
\begin{equation}\label{B1toB1M}
\sum_{m\le(x/2)^{1/k}} B_1 \Big( (x-m^k)^{\frac{1}{k}}\Big) =\sum_{M\in\mathcal{M}} \sum_{M<m\le M^{\prime}} B_1 \Big( (x-m^k)^{\frac{1}{k}}\Big)+O(\nu).
\end{equation}
By Lemma \ref{B1B2}, let $\alpha=(x-m^k)^{1/k}$, we have
\begin{align*}
\sum_{M<m\le M^{\prime}} B_1\Big((x-m^k)^{\frac{1}{k}}\Big)= &-\sum_{0<|h|\le H} \frac{T(M,h)}{2\pi ih} \\
&+O\Big(\sum_{M<m\le M^{\prime}} \min \Big(1, \frac{1}{H\Vert(x-m^k)^{\frac{1}{k}}\Vert}\Big) \Big),
\end{align*}
where
\begin{equation*}
T(M,h)=\sum_{M<m\le M^{\prime}} e\Big((x-m^k)^{\frac{1}{k}} h\Big).
\end{equation*}
Regarding the error term, still by Lemma \ref{B1B2}, we have
\begin{equation*}
\sum_{M<m\le M^{\prime}} \min \Big(1, \frac{1}{H\Vert(x-m^k)^{\frac{1}{k}}\Vert}\Big) = \sum_{h=-\infty}^{\infty} c(h) T(M,h),
\end{equation*}
so
\begin{equation}\label{MMprimeB1}
\sum_{M<m\le M^{\prime}} B_1\Big((x-m^k)^{\frac{1}{k}}\Big)=-\sum_{0<|h|\le H} \frac{T(M,h)}{2\pi ih} +O\Big(\sum_{h=-\infty}^{\infty} c(h) T(M,h)\Big).
\end{equation}
Now consider $T(M,h)$ when $h\ne0$. First, for convenience let
\begin{equation*}
f(\alpha)=h(x-\alpha^k)^{\frac{1}{k}},
\end{equation*}
then for $\alpha<x^{1/k}$,
\begin{equation*}
f^{\prime\prime}(\alpha)= -(k-1) h x \alpha^{k-2}(x-\alpha^k)^{\frac{1}{k}-2}.
\end{equation*}
So when $\alpha\in [M, M^{\prime}]$, we have
\begin{equation*}
(k-1) \vert h\vert x^{\frac{1}{k}-1} M^{k-2}\le \vert f^{\prime\prime}(\alpha) \vert \le 2^{k-\frac{1}{k}} (k-1) \vert h\vert x^{\frac{1}{k}-1} M^{k-2}.
\end{equation*}
Therefore, the function $f(\alpha)$ satisfies the conditions in Lemma \ref{vanderCorput}, where $a$, $b$, $\mu$, and $\eta$ are replaced by $M$, $M^{\prime}$, $(k-1) \vert h\vert x^{\frac{1}{k}-1} M^{k-2}$, and $2^{k-\frac{1}{k}}$ respectively. By that lemma, when $h\ne0$, we have
\begin{align*}
T(M,h) &\ll \big((k-1) \vert h\vert x^{\frac{1}{k}-1} M^{k-2}\big)^{-\frac{1}{2}} + (M^{\prime}-M) 2^{k-\frac{1}{k}} \big((k-1) \vert h\vert x^{\frac{1}{k}-1} M^{k-2}\big)^{\frac{1}{2}} \notag\\
&\ll \vert h\vert^{-\frac{1}{2}} x^{\frac{k-1}{2k}} M^{\frac{2-k}{2}} + \vert h\vert^{\frac{1}{2}} x^{\frac{1-k}{2k}} M^{\frac{k}{2}}.
\end{align*}
Finally, if $h=0$, then obviously $T(M, 0)\ll M$. Hence by (\ref{MMprimeB1}) and Lemma \ref{B1B2}, we have
\begin{align}\label{B1withH}
&\sum_{M<m\le M^{\prime}} B_1\Big((x-m^k)^{\frac{1}{k}}\Big) \notag\\
\ll& \sum_{0<|h|\le H} \frac{\vert T(M,h) \vert}{\vert h \vert} +\sum^{\infty}_{\substack{
h=-\infty\\
h\ne 0
}} \vert c(h) \vert \vert T(M,h) \vert + \vert c(0)\vert \vert T(M,0)\vert \notag\\
\ll& \sum_{0<|h|\le H} \big( \vert h\vert^{-\frac{3}{2}} x^{\frac{k-1}{2k}} M^{\frac{2-k}{2}} + \vert h\vert^{-\frac{1}{2}} x^{\frac{1-k}{2k}} M^{\frac{k}{2}} \big) \notag\\
&+\sum^{\infty}_{\substack{
h=-\infty\\
h\ne 0
}}\Big( \frac{1}{|h|}  \vert h\vert^{-\frac{1}{2}} x^{\frac{k-1}{2k}} M^{\frac{2-k}{2}} +\min\Big(\frac{1}{|h|}, \frac{H}{h^2}\Big) \vert h\vert^{\frac{1}{2}} x^{\frac{1-k}{2k}} M^{\frac{k}{2}} \Big) \notag\\
&+ \frac{2}{H}\Big( 1+\log \frac{H}{2}\Big) M \notag\\
\ll& x^{\frac{k-1}{2k}} M^{\frac{2-k}{2}} + H^{\frac{1}{2}} x^{\frac{1-k}{2k}} M^{\frac{k}{2}} + MH^{-1} \log (2H).
\end{align}
To minimize the size of the error terms, we need an optimal choice for $H$. A possible option is
\begin{equation}\label{Hvalue}
H=x^{\frac{3}{2k^2}},
\end{equation}
so that (\ref{B1withH}) becomes
\begin{equation*}
\sum_{M<m\le M^{\prime}} B_1\Big((x-m^k)^{\frac{1}{k}}\Big) \ll x^{\frac{k-1}{2k}} M^{\frac{2-k}{2}} + x^{\frac{3+2k-2k^2}{4k^2}} M^{\frac{k}{2}} + Mx^{-\frac{3}{2k^2}} \log x.
\end{equation*}
Therefore, by (\ref{B1toB1M}),
\begin{align}\label{B1FFinal}
&\sum_{m\le(x/2)^{1/k}} B_1 \Big( (x-m^k)^{\frac{1}{k}}\Big) \notag\\
\ll& \sum_{M\in\mathcal{M}} \big( x^{\frac{k-1}{2k}} M^{\frac{2-k}{2}}\big) + \sum_{M\in\mathcal{M}} \big(x^{\frac{3+2k-2k^2}{4k^2}} M^{\frac{k}{2}} \big)+ \sum_{M\in\mathcal{M}} \big(Mx^{-\frac{3}{2k^2}} \log x \big)+\nu \notag\\
\ll& x^{\frac{k-1}{2k}}\nu^{\frac{2-k}{2}}+x^{\frac{2k+3}{4k^2}}+x^{\frac{2k-3}{2k^2}}(\log x)+\nu.
\end{align}
From the first and the last terms, the only optimal choice for $\nu$ such that the error terms have the minimal size is 
\begin{equation}\label{nuvalue}
\nu=x^{\frac{k-1}{k^2}},
\end{equation}
which gives the conclusion.
\end{proof}
We notice that in the proof of Lemma \ref{lemmaerrorterm}, the first and the last terms of (\ref{B1FFinal}) are independent from the choice for $H$, so the optimal choice for $\nu$ does not depend on $H$ as well, and it is always supposed to be (\ref{nuvalue}). However, we do have plenty of choices for $H$ to replace (\ref{Hvalue}). For example, if $H=x^A$ where
\begin{equation*}
\frac{1}{k^2}<A<\frac{k-2}{k^2},
\end{equation*}
then we can also prove the lemma in a similar way.

Now we use Lemma \ref{lemmaerrorterm} to prove Theorem \ref{MainThm}. We start from the case when $s=2$.
\begin{lemma}\label{lemmawhens=2}
If $k\ge4$ is an integer and $x>1$, then
\begin{equation}\label{whens=2}
\sum_{n\le x} r_{k,2}(n) = \frac{\Gamma(\frac{k+1}{k})^2}{\Gamma(\frac{k+2}{k})} x^{\frac{2}{k}}-x^{\frac{1}{k}}+O\Big(x^{\frac{k-1}{k^2}} \Big).
\end{equation}
\end{lemma}
\begin{proof}
The LHS of (\ref{whens=2}) is the number of lattice points in Quadrant I under the curve $X^k+Y^k=x$. By separating the lattice points according to whether one or both coordinates are no more than $(x/2)^{1/k}$, we have
\begin{equation*}
\sum_{n\le x} r_{k,2}(n) = 2\sum_{m\le(x/2)^{1/k}} \big\lfloor(x-m^k)^{\frac{1}{k}}\big\rfloor -\Big\lfloor \Big( \frac{x}{2}\Big)^{\frac{1}{k}}\Big\rfloor^2.
\end{equation*}
By definition of $B_1(\alpha)$ and Lemma \ref{lemmaerrorterm}, we have
\begin{align}\label{rk,2nInitial}
\sum_{n\le x} r_{k,2}(n) &= 2\sum_{m\le(x/2)^{1/k}}\Big( (x-m^k)^{\frac{1}{k}}-\frac{1}{2}-B_1\big( (x-m^k)^{\frac{1}{k}}\big)\Big) -\Big\lfloor \Big( \frac{x}{2}\Big)^{\frac{1}{k}}\Big\rfloor^2 \notag \\
&= 2\sum_{m\le(x/2)^{1/k}} (x-m^k)^{\frac{1}{k}} - \Big\lfloor \Big( \frac{x}{2}\Big)^{\frac{1}{k}}\Big\rfloor - \Big\lfloor \Big( \frac{x}{2}\Big)^{\frac{1}{k}}\Big\rfloor^2 +O\Big(x^{\frac{k-1}{k^2}}\Big).
\end{align}
We write $(x-m^k)^{1/k}$ as an integral, and change the order of summation. The first term above is
\begin{align}\label{rk,2nStep2}
2\sum_{m\le(x/2)^{1/k}} (x-m^k)^{\frac{1}{k}} =& 2\sum_{m\le(x/2)^{1/k}} \int_m^{x^{1/k}} \alpha^{k-1} (x-\alpha^k)^{-\frac{k-1}{k}}\mathrm{d}\alpha\notag\\
=&2\int_0^{x^{1/k}} \min\Big( \lfloor\alpha\rfloor,\Big\lfloor\Big(\frac{x}{2}\Big)^{\frac{1}{k}}\Big\rfloor \Big) \alpha^{k-1} (x-\alpha^k)^{-\frac{k-1}{k}}\mathrm{d}\alpha\notag\\
=&\int_0^{(x/2)^{1/k}} 2\Big(\alpha-\frac{1}{2}\Big) \alpha^{k-1} (x-\alpha^k)^{-\frac{k-1}{k}}\mathrm{d}\alpha\notag\\
&-\int_0^{(x/2)^{1/k}} 2B_1(\alpha) \alpha^{k-1} (x-\alpha^k)^{-\frac{k-1}{k}}\mathrm{d}\alpha\notag\\
&+2\int_{(x/2)^{1/k}}^{x^{1/k}} \Big\lfloor\Big(\frac{x}{2}\Big)^{\frac{1}{k}}\Big\rfloor \alpha^{k-1} (x-\alpha^k)^{-\frac{k-1}{k}}\mathrm{d}\alpha.
\end{align}
A straightforward calculation on the last integral shows that
\begin{equation}\label{rk,2nStep2Term3}
\int_{(x/2)^{1/k}}^{x^{1/k}} \Big\lfloor\Big(\frac{x}{2}\Big)^{\frac{1}{k}}\Big\rfloor \alpha^{k-1} (x-\alpha^k)^{-\frac{k-1}{k}}\mathrm{d}\alpha= \Big\lfloor\Big(\frac{x}{2}\Big)^{\frac{1}{k}}\Big\rfloor \Big(\frac{x}{2}\Big)^{\frac{1}{k}}.
\end{equation}
Then we integrate the other two integrals by parts. The first integral is
\begin{align}\label{rk,2nStep2Term1}
&(1-2\alpha)(x-\alpha^k)^{\frac{1}{k}}\Big\vert_0^{(x/2)^{1/k}} -\int_0^{(x/2)^{1/k}} (x-\alpha^k)^{\frac{1}{k}} \mathrm{d}(1-2\alpha) \notag\\
=&-2\Big(\frac{x}{2}\Big)^{\frac{2}{k}}+\Big(\frac{x}{2}\Big)^{\frac{1}{k}}-x^{\frac{1}{k}}+2\int_0^{(x/2)^{1/k}} (x-\alpha^k)^{\frac{1}{k}}\mathrm{d}\alpha,
\end{align}
while the second one is
\begin{align*}
&\int_0^{(x/2)^{1/k}} 2\alpha^{k-1} (x-\alpha^k)^{-\frac{k-1}{k}}\mathrm{d}(B_2(\alpha)) \notag\\
=&2B_2\Big(\Big(\frac{x}{2}\Big)^{\frac{1}{k}}\Big)-2\int_0^{(x/2)^{1/k}} B_2(\alpha) \mathrm{d}\Big(\alpha^{k-1} (x-\alpha^k)^{-\frac{k-1}{k}}\Big).
\end{align*}
Since $B_2(\alpha)\ll 1$, and
\begin{equation*}
\frac{\mathrm{d}}{\mathrm{d}\alpha} \Big(\alpha^{k-1} (x-\alpha^k)^{-\frac{k-1}{k}}\Big)=x\alpha^{k-2}(x-\alpha^k)^{-\frac{2k-1}{k}} >0
\end{equation*}
when $0<\alpha<(x/2)^{1/k}$, we have
\begin{equation}\label{rk,2nStep2Term2}
\int_0^{(x/2)^{1/k}} 2B_1(\alpha) \alpha^{k-1} (x-\alpha^k)^{-\frac{k-1}{k}}\mathrm{d}\alpha=O(1).
\end{equation}
By (\ref{rk,2nInitial}), (\ref{rk,2nStep2}), (\ref{rk,2nStep2Term3}), (\ref{rk,2nStep2Term1}) and (\ref{rk,2nStep2Term2}), and the fact that
\begin{equation*}
\Big(\frac{x}{2}\Big)^{\frac{1}{k}}- \Big\lfloor\Big(\frac{x}{2}\Big)^{\frac{1}{k}}\Big\rfloor\ll 1
\end{equation*}
and its square
\begin{equation*}
\Big(\frac{x}{2}\Big)^{\frac{2}{k}}-2\Big\lfloor\Big(\frac{x}{2}\Big)^{\frac{1}{k}}\Big\rfloor \Big(\frac{x}{2}\Big)^{\frac{1}{k}}+\Big\lfloor\Big(\frac{x}{2}\Big)^{\frac{1}{k}}\Big\rfloor^2\ll 1,
\end{equation*}
we have
\begin{equation*}
\sum_{n\le x} r_{k,2}(n)=2\int_0^{(x/2)^{1/k}} (x-\alpha^k)^{\frac{1}{k}}\mathrm{d}\alpha-\Big(\frac{x}{2}\Big)^{\frac{2}{k}}-x^{\frac{1}{k}}+O\Big(x^{\frac{k-1}{k^2}}\Big).
\end{equation*}
Finally, the integral above represents the area of the set of points $(X, Y)$ in Quadrant I such that $X^k+Y^k\le x$ and $X\le(x/2)^{1/k}$, or equivalently, such that $Y^k+X^k\le x$ and $Y\le(x/2)^{1/k}$ when interchanging the variables $X$ and $Y$. Therefore, twice the integral represents the area of the whole region $X^k+Y^k\le x$ plus the area of the square $0\le X\le(x/2)^{1/k}, 0\le Y\le(x/2)^{1/k}$. Hence
\begin{equation*}
2\int_0^{(x/2)^{1/k}} (x-\alpha^k)^{\frac{1}{k}}\mathrm{d}\alpha-\Big(\frac{x}{2}\Big)^{\frac{2}{k}}=\int_0^{x^{1/k}} (x-\alpha^k)^{\frac{1}{k}}\mathrm{d}\alpha=\frac{1}{k} B\Big(\frac{1}{k},\frac{1}{k}+1\Big) x^{\frac{2}{k}}.
\end{equation*}
The lemma is then proved according to the relationship between the beta function and the gamma function.
\end{proof}
To finalize the proof of Theorem \ref{MainThm}, we use induction on $s$.  By Lemma \ref{lemmawhens=2}, the statement holds for the initial case $s=2$. Hence we only need to prove that if the statement is true for a particular $s$ where $2\le s\le k$, then it is also true for $s+1$. First, we notice that
\begin{equation*}
\sum_{m\le x} r_{k,s+1} (m)=\sum_{l\le x^{1/k}} \sum_{n \le x-l^{k}} r_{k,s} (n),
\end{equation*}
since both sides represent the number of $(s+1)$-tuples $(x_1, x_2, \cdots, x_s, l)$ such that all coordinates are positive integers and
\begin{equation*}
x_1^k+x_2^k+\cdots+x_s^k+l^k\le x.
\end{equation*}
Hence from assumption, we have
\begin{equation}\label{InductionMain}
\sum_{m\le x} r_{k,s+1}(m)= \frac{\Gamma \big( \frac{k+1}{k} \big)^s}{\Gamma\big( \frac{k+s}{k} \big)} \sum_{l\le x^{1/k}} (x-l^k)^{\frac{s}{k}} - \frac{s}{2}\cdot  \frac{\Gamma \big( \frac{k+1}{k} \big)^{s-1}}{\Gamma\big( \frac{k+s-1}{k} \big)} \sum_{l\le x^{1/k}} (x-l^k)^{\frac{s-1}{k}} + O\big( x^{\frac{sk-1}{k^2}} \big).
\end{equation}
Now we calculate the first sum on RHS. We have
\begin{align}\label{InductionMainSum1}
\sum_{l\le x^{1/k}} (x-l^k)^{\frac{s}{k}}  = & \sum_{l\le x^{1/k}} \int_l^{x^{1/k}} s\alpha^{k-1} (x-\alpha^k)^{\frac{s}{k}-1}\mathrm{d}\alpha\notag\\
= & \int_0^{x^{1/k}} \lfloor\alpha\rfloor \cdot s\alpha^{k-1} (x-\alpha^k)^{\frac{s}{k}-1} \mathrm{d}\alpha \notag\\
= & \int_0^{x^{1/k}} \Big( \alpha-\frac{1}{2} \Big)s\alpha^{k-1} (x-\alpha^k)^{\frac{s}{k}-1} \mathrm{d}\alpha \notag\\
& -\int_{\eta}^{x^{1/k}} B_1(\alpha) s\alpha^{k-1} (x-\alpha^k)^{\frac{s}{k}-1} \mathrm{d}\alpha \notag\\
& - \int_0^{\eta} B_1(\alpha) s\alpha^{k-1} (x-\alpha^k)^{\frac{s}{k}-1} \mathrm{d}\alpha ,
\end{align}
where
\begin{equation*}
\eta= \big(x-x^{\frac{k-1}{k}}\big)^{\frac{1}{k}}.
\end{equation*}
Integrating the first integral on RHS of (\ref{InductionMainSum1}) by parts, it is
\begin{equation*}
\int_0^{x^{1/k}}\Big(\frac{1}{2}-\alpha\Big) \mathrm{d}\big((x-\alpha^k)^{\frac{s}{k}}\big) = -\frac{1}{2} x^{\frac{s}{k}}+\int_0^{x^{1/k}} (x-\alpha^k)^{\frac{s}{k}} \mathrm{d}\alpha,
\end{equation*}
while the latter integral is
\begin{equation*}
\int_0^{x^{1/k}} (x-\alpha^k)^{\frac{s}{k}} \mathrm{d}\alpha = \frac{1}{k} B\Big( \frac{1}{k},\frac{s}{k}+1 \Big) x^{\frac{s+1}{k}} =\frac{\Gamma\big( \frac{k+1}{k}\big)\Gamma\big(\frac{k+s}{k} \big)}{\Gamma\big(\frac{k+s+1}{k} \big)} x^{\frac{s+1}{k}}.
\end{equation*}
Since $B_1(\alpha)\ll 1$, we have
\begin{equation*}
\int_{\eta}^{x^{1/k}} B_1(\alpha) s\alpha^{k-1} (x-\alpha^k)^{\frac{s}{k}-1} \mathrm{d}\alpha \ll \int_{\eta}^{x^{1/k}} \mathrm{d} \big(-(x-\alpha^k)^{\frac{s}{k}} \big) = x^{\frac{s(k-1)}{k^2}}.
\end{equation*}
Finally, we integrate the last integral of (\ref{InductionMainSum1}) by parts. As $B_1(\alpha)=\mathrm{d} (B_2(\alpha))$, the integral is
\begin{equation*}
B_2(\eta) s\eta^{k-1} (x-\eta^k)^{\frac{s}{k}-1} - \int_0^{\eta} B_2(\alpha) \mathrm{d} \big( s\alpha^{k-1} (x-\alpha^k)^{\frac{s}{k}-1} \big).
\end{equation*}
Since $B_2(\alpha)\ll 1$, we have
\begin{equation*}
B_2(\eta) s\eta^{k-1} (x-\eta^k)^{\frac{s}{k}-1} \ll \big( x^{\frac{1}{k}}\big)^{k-1} \big(x^{\frac{k-1}{k}}\big)^{\frac{s}{k}-1} = x^{\frac{s(k-1)}{k^2}}
\end{equation*}
and
\begin{equation*}
\int_0^{\eta} B_2(\alpha) \mathrm{d} \big( s\alpha^{k-1} (x-\alpha^k)^{\frac{s}{k}-1} \big) \ll \int_0^{\eta} \Big\vert\frac{\mathrm{d}}{\mathrm{d}\alpha} \big( s\alpha^{k-1} (x-\alpha^k)^{\frac{s}{k}-1} \big) \Big\vert\mathrm{d}\alpha.
\end{equation*}
We notice that if $0\le\alpha\le\eta$, $k\ge4$ and $2\le s\le k$, then
\begin{equation}\label{ImportantNonnegative}
\frac{\mathrm{d}}{\mathrm{d}\alpha} \big( s\alpha^{k-1} (x-\alpha^k)^{\frac{s}{k}-1} \big) = s\alpha^{k-2} (x-\alpha^k)^{\frac{s}{k}-2} \big((k-1)x+(1-s)\alpha^k\big) \ge0,
\end{equation}
so
\begin{equation*}
\int_0^{\eta} B_2(\alpha) \mathrm{d} \big( s\alpha^{k-1} (x-\alpha^k)^{\frac{s}{k}-1} \big) \ll  s\eta^{k-1} (x-\eta^k)^{\frac{s}{k}-1} \ll x^{\frac{s(k-1)}{k^2}}.
\end{equation*}
Hence
\begin{equation}\label{InductionSum1}
\sum_{l\le x^{1/k}} (x-l^k)^{\frac{s}{k}}= \frac{\Gamma\big( \frac{k+1}{k}\big)\Gamma\big(\frac{k+s}{k} \big)}{\Gamma\big(\frac{k+s+1}{k} \big)} x^{\frac{s+1}{k}} - \frac{1}{2} x^{\frac{s}{k}}+O\big( x^{\frac{s(k-1)}{k^2}} \big).
\end{equation}
Similarly,
\begin{equation}\label{InductionSum2}
\sum_{l\le x^{1/k}} (x-l^k)^{\frac{s-1}{k}}= \frac{\Gamma\big( \frac{k+1}{k}\big)\Gamma\big(\frac{k+s-1}{k} \big)}{\Gamma\big(\frac{k+s}{k} \big)} x^{\frac{s}{k}} - \frac{1}{2} x^{\frac{s-1}{k}} +O\big( x^{\frac{(s-1)(k-1)}{k^2}} \big).
\end{equation}
Therefore, by (\ref{InductionMain}), (\ref{InductionSum1}) and (\ref{InductionSum2}), we have
\begin{equation}
\sum_{m\le x} r_{k, s+1}(m)= \frac{\Gamma \big( \frac{k+1}{k} \big)^{s+1}}{\Gamma\big( \frac{k+s+1}{k} \big)}  x^{\frac{s+1}{k}} - \frac{(s+1)}{2}\cdot  \frac{\Gamma \big( \frac{k+1}{k} \big)^{s}}{\Gamma\big( \frac{k+s}{k} \big)} x^{\frac{s}{k}} + O\big( x^{\frac{sk-1}{k^2}} \big).
\end{equation}
which means that the statement is also true for $s+1$ under the assumption. Hence Theorem \ref{MainThm} is proved by induction.

\section{Several Notes}

Sometimes, we are particularly interested in the case when $s=k$. For convenience, we rewrite $r_{k,k}(n)$ as $r_k (n)$. Then we have
\begin{corollary}
If $k\ge4$ is an integer, then
\begin{equation}\label{CorResult}
\sum_{m=1}^n r_{k}(m) = \Gamma \Big( \frac{k+1}{k} \Big)^k n - \frac{k}{2}\cdot  \frac{\Gamma \big( \frac{k+1}{k} \big)^{k-1}}{\Gamma\big( \frac{2k-1}{k} \big)} n^{1-\frac{1}{k}} +O\Big(n^{\frac{k^2-k-1}{k^2}} \Big),
\end{equation}
or equivalently,
\begin{equation}\label{CorResultEq}
\sum_{m\le x} r_{k}(m) = \Gamma \Big( \frac{k+1}{k} \Big)^k x - \frac{k}{2}\cdot  \frac{\Gamma \big( \frac{k+1}{k} \big)^{k-1}}{\Gamma\big( \frac{2k-1}{k} \big)} x^{1-\frac{1}{k}} +O\Big(x^{\frac{k^2-k-1}{k^2}} \Big).
\end{equation}
\end{corollary}
We also notice that the inequality (\ref{ImportantNonnegative}) is not necessarily true when $s\ge k+1$ and $0\le\alpha\le\eta$. Therefore, we can not have any further induction from $s$ to $s+1$ when $s\ge k+1$. Hence Theorem \ref{MainThm} may not be true when $s$ is large.







\end{document}